\documentclass[11pt]{article}
\usepackage[]{amsmath,amssymb}
\usepackage{cite}
\newtheorem{theorem}{Theorem}[section]
\newtheorem{lemma}[theorem]{Lemma}
\newtheorem{proposition}[theorem]{Proposition}
\newtheorem{corollary}[theorem]{Corollary}
\newtheorem{exAux}[theorem]{Example}
\newenvironment{example}{\begin{exAux} \rm}{\end{exAux}}
\newtheorem{Def}[theorem]{Definition}
\newenvironment{definition}{\begin{Def} \rm}{\end{Def}}
\newtheorem{Note}[theorem]{Note}
\newenvironment{note}{\begin{Note} \rm}{\end{Note}}
\newtheorem{Claim}[theorem]{Claim}

\newtheorem{Problem}[theorem]{Problem}

\newtheorem{Rem}[theorem]{Remark}

\newtheorem{Not}[theorem]{Notation}

\newtheorem{Conj}[theorem]{Conjecture}
\newenvironment{conjecture}{\begin{Conj}}{\end{Conj}}
\newtheorem{Ass}[theorem]{Assumption}

\newenvironment{proof}{\medskip\noindent{\bf Proof.\ }}{\qed\medskip}
\newenvironment{proofof}[1]{\medskip\noindent{\bf Proof  of {#1}.\ 
}}{\qed\medskip}
\newcommand{\qed}{\hfill\mbox{$\Box$\qquad\qquad}}

\renewcommand{\th}{\theta}

\newcommand{\K}{\mathbb{K}}
\newcommand{\ot}{\otimes}

\newcommand{\ox}{{\overline{x}}}
\newcommand{\oy}{{\overline{y}}}

\renewcommand{\indent}{\hspace{6mm}}

\begin{document}
\thispagestyle{empty}

\begin{center}
\LARGE \bf
On the shape of a tridiagonal pair
\end{center}

\smallskip

\begin{center}
\Large
Kazumasa Nomura and 
Paul Terwilliger
\end{center}

\smallskip

\begin{quote}
\small 
\begin{center}
\bf Abstract
\end{center}
\indent
Let $\K$ denote a field and let $V$ denote a vector space over $\K$ with 
finite positive dimension.
We consider a pair of linear transformations $A:V \to V$
and $A^*:V \to V$ that satisfy the following conditions:
(i)
each of $A,A^*$ is diagonalizable;
(ii)
there exists an ordering $\lbrace V_i\rbrace_{i=0}^d$ of the eigenspaces of 
$A$ such that
$A^* V_i \subseteq V_{i-1} + V_{i} + V_{i+1}$ for $0 \leq i \leq d$,
where $V_{-1}=0$ and $V_{d+1}=0$;
(iii)
there exists an ordering $\lbrace V^*_i\rbrace_{i=0}^\delta$ 
of the eigenspaces of $A^*$ such that
$A V^*_i \subseteq V^*_{i-1} + V^*_{i} + V^*_{i+1}$ for
 $0 \leq i \leq \delta$,
where $V^*_{-1}=0$ and $V^*_{\delta+1}=0$;
(iv) 
there is no subspace $W$ of $V$ such that
$AW \subseteq W$, $A^* W \subseteq W$, $W \neq 0$, $W \neq V$.
We call such a pair a {\it tridiagonal pair} on $V$.
It is known that $d=\delta$ and for $0 \leq i \leq d$
the dimensions of $V_i$, $V^*_i$, $V_{d-i}$, $V^*_{d-i}$ coincide; 
we denote this common dimension by $\rho_i$.
In this paper we prove that $\rho_i \leq \rho_0  \binom{d}{i}$ 
for $0 \leq i \leq d$.
It is already known that $\rho_0=1$ if $\K$ is algebraically closed.


\end{quote}

\section{Introduction}

\indent
Throughout this paper $\K$ denotes a field.

\medskip

We begin by recalling the notion of a tridiagonal pair. 
We will use the following terms.
Let $V$ denote a vector space over $\K$ with finite positive dimension.
For a linear transformation $A:V\to V$ and a subspace $W \subseteq V$,
we call $W$ an {\em eigenspace} of $A$ whenever $W\not=0$ and there exists 
$\th \in \K$ such that $W=\{v \in V \,|\, Av = \th v\}$;
in this case $\th$ is the {\em eigenvalue} of $A$ associated with $W$.
We say that $A$ is {\em diagonalizable} whenever $V$ is spanned by the 
eigenspaces of $A$.

\begin{definition}  {\rm \cite[Definition 1.1]{ITT}} \label{def:tdp}
Let $V$ denote a vector space over $\K$ with finite positive dimension. 
By a {\em tridiagonal pair} on $V$ we mean an ordered pair of linear 
transformations $A:V \to V$ and $A^*:V \to V$ that satisfy the following 
four conditions.
\begin{enumerate}
\item 
Each of $A,A^*$ is diagonalizable.
\item 
There exists an ordering $\{V_i\}_{i=0}^d$ of the eigenspaces of $A$ 
such that 
\begin{equation}               \label{eq:t1}
A^* V_i \subseteq V_{i-1} + V_i+ V_{i+1} \qquad \qquad (0 \leq i \leq d),
\end{equation}
where $V_{-1} = 0$ and $V_{d+1}= 0$.
\item
There exists an ordering $\{V^*_i\}_{i=0}^{\delta}$ of the eigenspaces of 
$A^*$ such that 
\begin{equation}                \label{eq:t2}
A V^*_i \subseteq V^*_{i-1} + V^*_i+ V^*_{i+1} 
\qquad \qquad (0 \leq i \leq \delta),
\end{equation}
where $V^*_{-1} = 0$ and $V^*_{\delta+1}= 0$.
\item 
There does not exist a subspace $W$ of $V$ such  that $AW\subseteq W$,
$A^*W\subseteq W$, $W\not=0$, $W\not=V$.
\end{enumerate}
We say the pair $A,A^*$ is {\em over $\K$}.
\end{definition}

\begin{note}   \label{note:star}        \samepage
According to a common notational convention $A^*$ denotes 
the conjugate-transpose of $A$. We are not using this convention.
In a tridiagonal pair $A,A^*$ the linear transformations $A$ and $A^*$
are arbitrary subject to (i)--(iv) above.
\end{note}

\medskip

We refer the reader to 
\cite{AC1,AC2,AC3,Funk2,
ITT,IT:shape,IT:uqsl2hat,IT:nonnil,IT:Krawt,
IT:Drinfeld,IT:qRacah,IT:sharpen,IT:Mock,IT:aug,
N:aw,N:refine,N:heightone,
NT:split,NT:sharp,NT:towards,NT:structure,NT:muconj,NT:qRacah,
T:survey,Vidar} 
for background information about tridiagonal pairs.
See
\cite{AW,BIbook,B1,B2,B3,B4,BK1,BK2,BK3,BT1,BT2,BT3,CMT,Cau,Cur1,Cur2,
DR,Dav1,Dav2,Eg,Eld,Funk1,Go,GLZ,GH1,GH2,H1,H2,HT,
IT:qtet,IT:drg,IT:sl2loop,IT:qinverting,IT:A,ITW,
KoeSwa,L,Mik1,Mik2,Mik3,
NT:balanced,NT:formula,NT:det,NT:mu,NT:affine,NT:span,NT:switch,
Ros,Pas,Tanaka1,Tanaka2,Tang,T:subconst1,T:subconst2,
T:Leonard,T:24points,T:qSerre,T:intro,T:qRacah,T:split,T:array,
T:TD-D,T:survey,TV,V1,V2,Z1,Z2,God,Kim,Mik4,OS,S,AK,AK2,Koor,Koor2,KoorO,VZ}
for related topics.

\medskip

In order to motivate our results we recall a few facts about tridiagonal pairs.
Let $A,A^*$ denote a tridiagonal pair on $V$, as in Definition \ref{def:tdp}. 
By \cite[Lemma 4.5]{ITT} the integers $d$ and $\delta$ from (ii), (iii) are 
equal; we call this common value the {\em diameter} of the pair.
An ordering of the eigenspaces of $A$ (resp. $A^*$)
is said to be {\em standard} whenever it satisfies
\eqref{eq:t1} (resp. \eqref{eq:t2}). We comment on the uniqueness of the
standard ordering. Let $\{V_i\}_{i=0}^d$ denote a standard
ordering of the eigenspaces of $A$. By \cite[Lemma 2.4]{ITT}, the ordering
$\{V_{d-i}\}_{i=0}^d$ is also standard and no further ordering
is standard. A similar result holds for the eigenspaces
of $A^*$. Let $\{V_i\}_{i=0}^d$ (resp. $\{V^*_i\}_{i=0}^d$) 
denote a standard ordering of the eigenspaces of $A$ (resp. $A^*$).
By \cite[Corollary 5.7]{ITT}, for $0 \leq i \leq d$ the spaces $V_i$, $V^*_i$
have the same dimension; we denote this common dimension by $\rho_i$. 
By \cite[Corollaries 5.7, 6.6]{ITT} the sequence $\{\rho_i\}_{i=0}^d$ is 
symmetric and unimodal;
that is $\rho_i=\rho_{d-i}$ for $0 \leq i \leq d$ and
$\rho_{i-1} \leq \rho_i$ for $1 \leq i \leq d/2$.
We call the sequence $\{\rho_i\}_{i=0}^d$ the {\em shape} of $A,A^*$.
We now state our main result.

\medskip

\begin{theorem}   \label{thm:main} \samepage
The shape $\{\rho_i\}_{i=0}^d$ of a tridiagonal pair satisfies
\[
  \rho_i \leq  \rho_0 \binom{d}{i}  \qquad\qquad (0 \leq i \leq d).
\]
\end{theorem}

\medskip

Consider the scalar $\rho_0$ in Theorem \ref{thm:main}. 
By \cite[Theorem 1.3]{NT:structure}, 
if the tridiagonal pair is over an algebraically closed field then $\rho_0 =1$. 

\medskip

\begin{corollary}    \label{cor:main}
Let $\{\rho_i\}_{i=0}^d$ denote the shape of a tridiagonal pair over an 
algebraically closed field. Then
\[
   \rho_i \leq \binom{d}{i} \qquad\qquad        (0 \leq i \leq d).
\]
\end{corollary}

\medskip

We have some comments on Theorem \ref{thm:main} and Corollary \ref{cor:main}. 
As far as we know, the first reference in the literature to anything 
resembling these results appears in \cite{ITT}, where Corollary \ref{cor:main}
appears as a conjecture.
Between then and now the conjecture was proven in the following special cases. 
Associated with a tridiagonal pair is a parameter $q$ that is used to describe 
the eigenvalues.
There is a family of tridiagonal pairs with $q=1$ said to have 
{\em Krawtchouk type} \cite{IT:Krawt}.
There is a family of tridiagonal pairs with $q \neq 1$,
$q \neq -1$ said to have {\em $q$-geometric type} 
\cite{IT:uqsl2hat,IT:nonnil,IT:qinverting} or
{\em $q$-Serre type} \cite{AC1,AC2,AC3}. 
The most general family of tridiagonal pairs with $q \neq 1$,
$q \neq -1$ is said to have {\em $q$-Racah  type} \cite{IT:qRacah}.
The first progress on proving this conjecture appeared in \cite{IT:shape}.
In that paper the conjecture was proven for $q$-geometric type, 
assuming that $\K$ has characteristic $0$.
Then in \cite{IT:Krawt} the conjecture was proven for Krawtchouk type, 
assuming $\K$ has characteristic $0$.
Then in \cite{IT:aug} the conjecture was proven for $q$ not a root of 
unity, assuming $\K$ has characteristic $0$.
In \cite{IT:qRacah} the conjecture was proven for $q$-Racah type.
In our proof of Theorem \ref{thm:main} we do not formally use the above 
partial results. 
Our proof is case-free and does not refer to $q$. However we would like to 
acknowledge that the intuition gained by studying the above partial results 
was invaluable in finding our case-free approach.
We would also like to acknowledge many conversations between T. Ito and 
the second author on the general subject of this paper.

\medskip

Our proof of Theorem \ref{thm:main} is contained in Section
\ref{sec:proofofmain}.

\section{Tridiagonal systems}

\indent
When working with a tridiagonal pair, it is often convenient to consider
a closely related object called a tridiagonal system.
To define a tridiagonal system, we recall a few concepts from linear algebra.
Let $V$ denote a vector space over $\K$ with finite positive dimension.
Let ${\rm End}(V)$ denote the $\K$-algebra of all linear
transformations from $V$ to $V$.
Let $A$ denote a diagonalizable element of $\text{End}(V)$.
Let $\{V_i\}_{i=0}^d$ denote an ordering of the eigenspaces of $A$
and let $\{\th_i\}_{i=0}^d$ denote the corresponding ordering of the 
eigenvalues of $A$.
For $0 \leq i \leq d$ define $E_i \in \text{End}(V)$ such that 
$(E_i-I)V_i=0$ and $E_iV_j=0$ for $j \neq i$ $(0 \leq j \leq d)$.
Here $I$ denotes the identity of $\mbox{\rm End}(V)$.
We call $E_i$ the {\em primitive idempotent} of $A$ corresponding to $V_i$
(or $\th_i$).
Observe that
(i) $I=\sum_{i=0}^d E_i$;
(ii) $E_iE_j=\delta_{i,j}E_i$ $(0 \leq i,j \leq d)$;
(iii) $V_i=E_iV$ $(0 \leq i \leq d)$;
(iv) $A=\sum_{i=0}^d \theta_i E_i$.
Moreover
\begin{equation}         \label{eq:defEi}
  E_i=\prod_{\stackrel{0 \leq j \leq d}{j \neq i}}
          \frac{A-\theta_jI}{\theta_i-\theta_j}.
\end{equation}
Note that each of $\{A^i\}_{i=0}^d$, $\{E_i\}_{i=0}^d$ is a basis for the 
$\K$-subalgebra of $\mbox{\rm End}(V)$ generated by $A$.

\medskip

Now let $A,A^*$ denote a tridiagonal pair on $V$.
An ordering of the primitive idempotents or eigenvalues of $A$ (resp. $A^*$)
is said to be {\em standard} whenever the corresponding ordering of the 
eigenspaces of $A$ (resp. $A^*$) is standard.

\medskip

\begin{definition} \cite[Definition 2.1]{ITT} \label{def:TDsystem}
Let $V$ denote a vector space over $\K$ with finite positive dimension.
By a {\em tridiagonal system} on $V$ we mean a sequence
\[
 \Phi=(A;\{E_i\}_{i=0}^d;A^*;\{E^*_i\}_{i=0}^d)
\]
that satisfies (i)--(iii) below.
\begin{itemize}
\item[(i)]
$A,A^*$ is a tridiagonal pair on $V$.
\item[(ii)]
$\{E_i\}_{i=0}^d$ is a standard ordering
of the primitive idempotents of $A$.
\item[(iii)]
$\{E^*_i\}_{i=0}^d$ is a standard ordering
of the primitive idempotents of $A^*$.
\end{itemize}
We say $\Phi$ is {\em over} $\K$.
We call $V$ the {\em vector space underlying $\Phi$}.
\end{definition}

\medskip

The following result is immediate from lines (\ref{eq:t1}),  (\ref{eq:t2})
and Definition \ref{def:TDsystem}.

\medskip

\begin{lemma} {\rm \cite[Lemma 2.5]{NT:towards}}  \label{lem:trid}  \samepage
Let $(A;\{E_i\}_{i=0}^d;A^*;\{E^*_i\}_{i=0}^d)$ denote a tridiagonal system.
Then for $0 \leq i,j,k \leq d$ the following {\rm (i)}, {\rm (ii)} hold.
\begin{itemize}
\item[\rm (i)]
$E^*_i A^k E^*_j =0\;\;$ if $k<|i-j|$.
\item[\rm (ii)]
$E_i A^{*k} E_j = 0\;\;$ if  $k<|i-j|$.
\end{itemize}
\end{lemma}

\begin{definition}        \label{def}        \samepage
Let $\Phi=(A;\{E_i\}_{i=0}^d;A^*;\{E^*_i\}_{i=0}^d)$ denote a tridiagonal 
system on $V$.
For $0 \leq i \leq d$ let $\th_i$ (resp. $\th^*_i$) denote the eigenvalue of 
$A$ (resp. $A^*$) associated with the eigenspace $E_iV$ (resp. $E^*_iV$).
We emphasize that the scalars $\{\th_i\}_{i=0}^d$, $\{\th^*_i\}_{i=0}^d$
are contained in $\K$, and that
\begin{equation}    \label{eq:distinct}
  \th_i \neq \th_j, \qquad \th^*_i \neq \th^*_j
  \qquad\qquad \text{ if $i \neq j$ } \qquad (0 \leq i,j \leq d).
\end{equation}
We call $\{\th_i\}_{i=0}^d$ (resp. $\{\th^*_i\}_{i=0}^d$)
the {\em eigenvalue sequence} (resp. {\em dual eigenvalue sequence}) of $\Phi$.
\end{definition}

\begin{lemma}  {\rm \cite[Theorem 11.1]{ITT}}  \label{lem:corec} \samepage 
With reference to Definition {\rm \ref{def}},
the expressions
\begin{equation}    \label{eq:indep}
 \frac{\th_{i-2}-\th_{i+1}}{\th_{i-1}-\th_{i}},  \qquad\qquad
 \frac{\th^*_{i-2}-\th^*_{i+1}}{\th^*_{i-1}-\th^*_{i}}
\end{equation}
are equal and independent of $i$ for $2 \leq i \leq d-1$.
\end{lemma}

\section{The algebra $T$}

\indent
In our proof of Theorem \ref{thm:main} we will make heavy use of 
the following algebra.

\medskip

\begin{definition}{\rm \cite[Definition 2.4]{NT:structure}}\label{def:T}
\samepage
Let $d$ denote a nonnegative integer. Let
$(\{\th_i\}_{i=0}^d$; $\{\th^*_i\}_{i=0}^d)$ denote a sequence of scalars
taken from $\K$ that satisfies \eqref{eq:distinct} and
Lemma {\rm \ref{lem:corec}}.
Let $T$ denote the associative $ \K$-algebra with $1$, defined by generators
$a$, $\{e_i\}_{i=0}^d$, $a^*$, $\{e^*_i\}_{i=0}^d$ and the following relations.
\begin{equation}                            \label{eq:eiej}
  e_ie_j=\delta_{i,j}e_i, \qquad 
  e^*_ie^*_j=\delta_{i,j}e^*_i \qquad\qquad 
  (0 \leq i,j \leq d),
\end{equation}
\begin{equation}                            \label{eq:sumei}
  1=\sum_{i=0}^d e_i, \qquad\qquad
  1=\sum_{i=0}^d e^*_i,
\end{equation}
\begin{equation}                                  \label{eq:sumthiei}
   a = \sum_{i=0}^d \theta_ie_i, \qquad\qquad
   a^* = \sum_{i=0}^d \theta^*_i e^*_i,
\end{equation}
\begin{equation}                                     \label{eq:esiakesj}
 e^*_i a^k e^*_j = 0  \qquad \text{if $\;k<|i-j|$}
                  \qquad\qquad (0 \leq i,j,k \leq d),
\end{equation}
\begin{equation}                                      \label{eq:eiaskej}
 e_i {a^*}^k e_j = 0  \qquad \text{if $\;k<|i-j|$} 
                  \qquad\qquad (0 \leq i,j,k \leq d). 
\end{equation}
We call $\{e_i\}_{i=0}^d$ and $\{e^*_i\}_{i=0}^d$ the 
{\it idempotent generators} for $T$.
\end{definition}

\medskip

The algebra $T$ is related to tridiagonal systems as follows.

\medskip

\begin{lemma}  {\rm \cite[Lemma 2.5]{NT:structure}} \label{lem:T-module}
Let $(A;\{E_i\}_{i=0}^d;A^*;\{E^*_i\}_{i=0}^d)$ denote a tridiagonal
system on $V$ with eigenvalue sequence
$\{\th_i\}_{i=0}^d$ and dual eigenvalue sequence $\{\th^*_i\}_{i=0}^d$.
Let $T$ denote the $\K$-algebra from Definition {\rm \ref{def:T}}
corresponding to 
$(\{\th_i\}_{i=0}^d; \{\th^*_i\}_{i=0}^d)$.
Then there exists a unique $T$-module structure on $V$ such that
$a$, $a^*$, $e_i$, $e^*_i$ acts on $V$ as
$A$, $A^*$, $E_i$, $E^*_i$, respectively.
Moreover this $T$-module is irreducible.
\end{lemma}

\begin{definition}     \samepage    \label{def:D}
With reference to Definition {\rm \ref{def:T}} we denote by
$D$ (resp. $D^*$)  the $ \K$-subalgebra
of $T$ generated by $a$ (resp. $a^*$).
\end{definition}

\begin{lemma} {\rm \cite[Lemma 4.2]{NT:muconj}} \label{lem:basisD} \samepage
With reference to Definition {\rm \ref{def:T}} the following
{\rm (i)}, {\rm (ii)} hold.
\begin{itemize}
\item[\rm (i)]
Each of $\{a^i\}_{i=0}^d$, $\{e_i\}_{i=0}^d$ is a basis for $D$.
\item[\rm (i)]
Each of $\{{a^*}^i\}_{i=0}^d$, $\{e^*_i\}_{i=0}^d$ is a basis for $D^*$.
\end{itemize}
\end{lemma}

\medskip

The next result is about $D$; a similar result holds for $D^*$.

\medskip

\begin{lemma} {\rm \cite[Lemma 3.4]{NT:structure}} \label{lem:replace}  
\samepage
With reference to Definition {\rm \ref{def:T}} consider 
the following basis for $D$:
\begin{equation}   \label{eq:E0Ed}
  e_0,e_1,\ldots,e_d.
\end{equation}
For $0 \leq n \leq d$, if we replace any $(n+1)$-subset of \eqref{eq:E0Ed}
by $1,a,a^2, \ldots, a^n$ then the result is still a basis for $D$.
\end{lemma}

\section{The space $D^* \ot D \ot D^*$}
\label{sec:DDD}

\indent
Throughout this section the following notation will be in effect.
Fix a nonnegative integer $d$. 
Fix a sequence $(\{\th_i\}_{i=0}^d;\{\th^*_i\}_{i=0}^d)$  of scalars 
taken from $\K$ that satisfies \eqref{eq:distinct} and Lemma \ref{lem:corec}.
Let $T$ denote the corresponding algebra from Definition \ref{def:T}.

In our investigation of $T$ we will initially study the linear dependencies 
among elements of the form
  $e^*_s e_i e^*_t $ $(0 \leq s,i,t \leq d)$.
In order to do this in a systematic way, 
we will analyze the vector space $D^* \ot D \ot D^*$,
where $\ot$ means $\ot_{\K}$. 

\medskip

\begin{definition}    \label{def:R3}
Let $R=R^{(3)}$ denote the subspace of $D^* \ot D \ot D^*$ spanned by
\begin{equation}        \label{eq:R1}
 \{e^*_s \ot a^k \ot e^*_t
      \;|\; 0 \leq s,t,k \leq d,\; k < |s-t| \}.
\end{equation}
We note that the elements \eqref{eq:R1} form a basis for $R$.
\end{definition}

\medskip

We mention one significance of the space $R$.

\medskip

\begin{lemma}    \label{lem:piR1}   \samepage
Consider the $\K$-linear transformation
\[
 \begin{array}{ccc}
    D^* \ot D \ot D^* & \qquad \to \qquad & T \\
    X \ot Y \ot Z  & \qquad \mapsto \qquad & XYZ
 \end{array}
\]
Then the kernel of this map contains the space $R$ from 
Definition {\rm \ref{def:R3}}.
\end{lemma}

\begin{proof}
Follows from \eqref{eq:esiakesj}.
\end{proof}

\medskip

We now define some elements of $D^* \ot D \ot D^*$ said to be {\em zigzag}.
We will show that the zigzag elements in $D^* \ot D \ot D^*$ form a 
basis for a complement of $R$ in $D^* \ot D \ot D^*$.

\medskip

\begin{definition}    \label{def:between}     \samepage
For integers $i,r,j$ we say $r$ is {\em between} the ordered pair $i,j$
whenever $i \geq r>j$ or $i \leq r<j$.
\end{definition}

\begin{definition}   \label{def:zigzagelement1}      \samepage
An element of $D^* \ot D \ot D^*$ is said to be {\em zigzag} whenever it has 
the form $e^*_s \ot e_i \ot e^*_t$ $(0 \leq s,i,t\leq d)$ and $i$ is 
not between $s,t$.
\end{definition}

\begin{theorem}   \label{thm:basisDDD}  \samepage
The following elements {\rm (i)}, {\rm (ii)} together form a basis
for $D^* \ot D \ot D^*$.
\begin{itemize}
\item[\rm (i)]
the elements \eqref{eq:R1};
\item[\rm (ii)]
the zigzag elements in $D^* \ot D \ot D^*$.
\end{itemize}
\end{theorem}

\begin{proof}
Observe
\[
   D^* \ot D \ot D^* = \sum_{s=0}^d \sum_{t=0}^d e^*_s \ot D \ot e^*_t
     \qquad\qquad \text{(direct sum)}.
\]
Lemma \ref{lem:replace} implies that for $0 \leq s,t\leq d$ the following two
sets together form a basis for $D$:
\[
   \{a^k \,|\, 0 \leq k \leq d, \; k <|s-t|\},
\]
\[
   \{e_i \,|\, 0 \leq i \leq d, \; \text{ $i$ is not between $s,t$}\}.
\]
Therefore the following two sets together form
a basis for $e^*_s \ot D \ot e^*_t$:
\[
   \{e^*_s \ot a^k \ot e^*_t \,|\, 0 \leq k \leq d,\; k<|s-t|\},
\]
\[
   \{e^*_s \ot e_i \ot e^*_t \,|\, 0 \leq i \leq d,\; 
                 \text{ $e^*_s \ot e_i \ot e^*_t$ is zigzag} \}.
\]
The result follows.
\end{proof}

\begin{corollary}   \label{cor:zigzagDDD}  \samepage
The zigzag elements in $D^* \ot D \ot D^*$ form a basis for a complement 
of $R$ in $D^* \ot D \ot D^*$.
\end{corollary}

\section{The space $D^* \ot D \ot D^* \ot D$}
\label{sec:R}

\indent
For this section and the next, the following notation will be in effect.
Fix a nonnegative integer $d$. 
Fix a sequence $(\{\th_i\}_{i=0}^d;\{\th^*_i\}_{i=0}^d)$  of scalars 
taken from $\K$ that satisfies \eqref{eq:distinct} and Lemma \ref{lem:corec}.
Let $T$ denote the corresponding algebra from Definition \ref{def:T}.

In Section \ref{sec:DDD} we used $D^* \ot D \ot D^*$ to describe some linear 
dependencies among elements of $T$ of the form $e^*_s e_i e^*_t$ 
$(0 \leq s,i,t\leq d)$.
An analysis of $D \ot D^* \ot D$ yields similar dependencies among elements 
of $T$ of the form $e_s e^*_i e_t$ $(0 \leq s,i,t\leq d)$.
We now investigate linear dependencies among elements of $T$ of the 
form $e^*_s e_i e^*_j e_t$ $(0 \leq s,i,j,t\leq d)$.
To this end we will analyze the vector space $D^* \ot D \ot D^*\ot D$. 
This analysis will take up to the end of Section \ref{sec:ZZDDDD}.

\medskip

\begin{definition}        \label{def:R}      \samepage
Let $R=R^{(4)}$ denote the subspace of $D^* \ot D \ot D^* \ot D$
spanned by the union of the following two subsets:
\begin{equation}        \label{eq:Rright}
  \{e^*_s \ot e_i \ot {a^*}^k \ot e_t 
      \;|\; 0 \leq s,t,i,k \leq d,\; k <|i-t| \},
\end{equation}
\begin{equation}        \label{eq:Rleft}
 \{e^*_s \ot a^k \ot e^*_j \ot e_t 
      \;|\; 0 \leq s,t,j,k \leq d,\; k < |s-j| \}.
\end{equation}
\end{definition}

\medskip

We mention one significance of the space $R$.

\medskip

\begin{lemma}    \label{lem:piR}   \samepage
Consider the $\K$-linear transformation
\[
 \begin{array}{ccc}
    D^* \ot D \ot D^* \ot D & \qquad \to \qquad & T \\
    X \ot Y \ot Z \ot W  & \qquad \mapsto \qquad 
         & XYZW
 \end{array}
\]
Then the kernel of this map contains the space $R$ from Definition 
{\em \ref{def:R}}.
\end{lemma}

\begin{proof}
Follows from \eqref{eq:esiakesj} and \eqref{eq:eiaskej}.
\end{proof}

\medskip

It will be convenient to decompose 
$D^* \ot D \ot D^* \ot D$ as follows.

\medskip

\begin{lemma}   \label{lem:directsum}   \samepage
We have
\begin{equation}    \label{eq:essDDset}
 D^* \ot D \ot D^* \ot D
 = \sum_{s=0}^d \sum_{t=0}^d e^*_s \ot D \ot D^* \ot e_t
   \qquad\qquad \text{\rm (direct sum)}.
\end{equation}
\end{lemma}

\begin{definition}         \label{def:Rst}   \samepage
With reference to Definition {\rm \ref{def:R}},
for $0 \leq s,t \leq d$ let $R_{s,t}$ denote the subspace of $R$
spanned by the union of the following two subsets:
\begin{equation}         \label{eq:Rstright}
  \{e^*_s \ot e_i \ot {a^*}^k \ot e_t \;|\;  0 \leq i,k \leq d,\; k<|i-t|\},
\end{equation}
\begin{equation}         \label{eq:Rstleft}
  \{e^*_s \ot a^k \ot e^*_j \ot e_t \;|\; 0 \leq j,k \leq d,\; k<|s-j|\}.
\end{equation}
\end{definition}

\begin{lemma}        \label{lem:sumRst}    \samepage
With reference to Definition {\rm \ref{def:Rst}},
\begin{equation}
 R = \sum_{s=0}^d \sum_{t=0}^d R_{s,t}  \qquad\qquad \text{\rm (direct sum)}.
\end{equation}
\end{lemma}

\begin{lemma}       \label{lem:intersection}       \samepage
With reference to Definition {\rm \ref{def:Rst}},
for $0 \leq s,t \leq d$ the space $R_{s,t}$ is equal to the intersection
of $R$ and $e^*_s \ot D \ot D^* \ot e_t$.
\end{lemma}

\begin{lemma}   \label{lem:sizeRst}   \samepage
The following {\rm (i)}, {\rm (ii)} hold.
\begin{itemize}
\item[\rm (i)]
The number of elements in \eqref{eq:Rstright} is 
$d(d+1)/2-t(d-t)$.
\item[\rm (ii)]
The number of elements in \eqref{eq:Rstleft} is 
$d(d+1)/2-s(d-s)$.
\end{itemize}
\end{lemma}

\section{The zigzag elements of $D^* \ot D \ot D^* \ot D$}
\label{sec:ZZDDDD}

\indent
In this section we define some elements of $D^* \ot D \ot D^* \ot D$ 
said to be {\em zigzag}. 
We will prove that \eqref{eq:Rright}, \eqref{eq:Rleft}  together form 
a basis for $R$, 
and that the zigzag elements form a basis for a complement of R in 
$D^* \ot D \ot D^* \ot D$.

\medskip

\begin{definition}   \label{def:zigzagelement}      \samepage
An element of $D^* \ot D \ot D^* \ot D$ is said to be {\em zigzag} whenever 
it has the form $e^*_s \ot e_i \ot e^*_j \ot e_t$ $(0 \leq s,i,j,t\leq d)$
and the following (i)--(iii) hold.
\begin{itemize}
\item[(i)] $i$ is not between $s,j$.
\item[(ii)] $j$ is not between $i,t$.
\item[(iii)] At least one of $i,j$ is not between $s,t$.
\end{itemize}
We say $e^*_s \ot e_i \ot e^*_j \ot e_t$
is {\em nonzigzag} whenever it is not zigzag.
\end{definition}

\begin{lemma}        \label{lem:ZZ} 
For $0 \leq s,t,i,j \leq d$,
the element $e^*_s \ot e_i \ot e^*_j \ot e_t$
is zigzag if and only if at least one of the following conditions holds.
\[
\begin{array} {c|c}
 \text{\rm case} & \text{\rm conditions}
\\ \hline \hline
 s<t & 0 \leq i<s  \quad  \text{and} \quad  t \leq j \leq d
\\
     &  s<i<t  \quad  \text{and} \quad  0 \leq j <s
\\
     & t \leq i \leq d \quad  \text{and} \quad  0 \leq j \leq t
\\ \hline
 s=t & 0 \leq i<s  \quad \text{and} \quad  s \leq j \leq d
\\
     & i = s  \quad \text{and} \quad  j=s
\\
     &   s < i \leq d \quad  \text{and} \quad  0 \leq j \leq s
\\ \hline
 s>t & 0 \leq i \leq t \quad  \text{and} \quad  t \leq j \leq d
\\
     & t < i < s  \quad  \text{and} \quad  s < j \leq d
\\
     & s < i \leq d  \quad  \text{and} \quad  0 \leq j \leq t
\end{array}
\]
\end{lemma}

\begin{example}  \label{exam:1}
Assume $d=9$.
We display the zigzag elements of 
$e^*_s \ot D \ot D^* \ot e_t$ for several values of $s,t$. 
For each value of $s,t$ a table is given.
For each table and for $0 \leq i,j\leq d$,
the $(i,j)$-entry is `z' if
and only if the element $e^*_s \ot e_i \ot e^*_j \ot e_t$ is zigzag.
\begin{itemize}
\item[(i)]
Assume $s=3$ and $t=6$.
\begin{center}
\begin{tabular}{cc|cccccccccc}
  \multicolumn{5}{c}{}      & $s$ &     &     & $t$
\\
    &    & $0$ & $1$ & $2$ & $3$ & $4$ & $5$ & $6$ & $7$ & $8$ & $9$
\\ \cline{2-12}
    & $0$ &  &  &  &  &  &  & z & z & z & z
\\
    & $1$ &  &  &  &  &  &  & z & z & z & z
\\
    & $2$ &  &  &  &  &  &  & z & z & z & z
\\
$s$ & $3$ 
\\
    & $4$ & z & z & z
\\
    & $5$ & z & z & z
\\
$t$ & $6$ & z & z & z & z & z & z & z
\\
    & $7$ & z & z & z & z & z & z & z
\\
    & $8$ & z & z & z & z & z & z & z
\\
    & $9$ & z & z & z & z & z & z & z
\end{tabular}
\end{center}
\item[(ii)] 
Assume $s=4$ and $t=4$.
\begin{center}
\begin{tabular}{cc|cccccccccc}
    \multicolumn{6}{c}{}          & $s$
\\    
    &     & $0$ & $1$ & $2$ & $3$ & $4$ & $5$ & $6$ & $7$ & $8$ & $9$
\\ \cline{2-12}
    & $0$ &  &  &  &  & z & z & z & z & z & z
\\
    & $1$ &  &  &  &  & z & z & z & z & z & z
\\
    & $2$ &  &  &  &  & z & z & z & z & z & z
\\
    & $3$ &  &  &  &  & z & z & z & z & z & z
\\
$s$ & $4$ &  &  &  &  & z
\\
    & $5$ & z & z & z & z & z
\\
    & $6$ & z & z & z & z & z
\\
    & $7$ & z & z & z & z & z
\\
    & $8$ & z & z & z & z & z
\\
    & $9$ & z & z & z & z & z
\end{tabular}
\end{center}
\item[(iii)] 
Assume $s=6$ and $t=3$.
\begin{center}
\begin{tabular}{cc|cccccccccc}
 \multicolumn{2}{c}{} & &   &     & $t$ &     &     & $s$
\\
    &     & $0$ & $1$ & $2$ & $3$ & $4$ & $5$ & $6$ & $7$ & $8$ & $9$
\\ \cline{2-12}
    & $0$ &  &  &  & z & z & z & z & z & z & z
\\
    & $1$ &  &  &  & z & z & z & z & z & z & z
\\
    & $2$ &  &  &  & z & z & z & z & z & z & z
\\
$t$ & $3$ &  &  &  & z & z & z & z & z & z & z
\\
    & $4$ &  &  &  &   &   &   &   & z & z & z
\\
    & $5$ &  &  &  &   &   &   &   & z & z & z
\\
$s$ & $6$
\\
    & $7$ & z & z & z & z
\\
     &$8$ & z & z & z & z
\\
    & $9$ & z & z & z & z
\end{tabular}
\end{center}
\end{itemize}
\end{example}

\begin{lemma}     \label{lem:sizeZZ}   \samepage
For $0 \leq s,t \leq d$, the number of zigzag elements in
$e^*_s \ot D \ot D^* \ot e_t$ is
$d+1+s(d-s)+t(d-t)$.
\end{lemma}

\begin{proof}
Routine using Lemma \ref{lem:ZZ}.
\end{proof}

\medskip

From now until Theorem \ref{thm:mainn4} we fix integers $s,t$ 
$(0 \leq s,t\leq d)$.
Our next goal is to prove that \eqref{eq:Rstright}, \eqref{eq:Rstleft} 
together form a basis for $R_{s,t}$, and that the zigzag elements in 
$e^*_s \ot D \ot D^* \ot e_t$ form a basis for a complement of 
$R_{s,t}$ in $e^*_s \ot D \ot D^* \ot e_t$.

\medskip

We partition the set of nonzigzag elements into two subsets,
consisting of the {\em positive} and {\em negative} elements.
These elements are defined as follows.

\medskip

\begin{definition}        \label{def:positive}
For $0 \leq i,j\leq d$,
the element $e^*_s \ot e_i \ot e^*_j \ot e_t$
is said to be {\em positive} whenever it satisfies
at least one of the following conditions:
\[
\begin{array} {c|c}
 \text{case} &  \text{conditions}
\\ \hline \hline
 s<t & \; 0 \leq i < s \; \text{ and } \; i \leq j <t
\\
     & \; s \leq i < t \; \text{ and } \; s \leq j < s+t-i
\\
     & \; t < j \leq i \leq d
\\ \hline
 s=t & \; 0 \leq i \leq j < s
\\
     & \; s < j \leq i \leq d
\\ \hline
 s>t & \; 0 \leq i \leq j < t
\\
     & \; t < i \leq s \; \text{ and } \; s+t-i < j \leq s
\\
     & \; s < i \leq d \; \text{ and } \; t < j \leq i
\end{array}
\]
Observe that if $e^*_s \ot e_i \ot e^*_j \ot e_t$ is positive 
then it is nonzigzag. We call the element
$e^*_s \ot e_i \ot e^*_j \ot e_t$ {\em negative} whenever it is 
nonzigzag and not positive.
\end{definition}

\begin{lemma}         \label{lem:i-t}      \samepage
The following {\rm (i)}, {\rm (ii)} hold.
\begin{itemize}
\item[\rm (i)]
For $0 \leq i \leq d$ the number of $j$ $(0 \leq j \leq d)$ such that
$e^*_s \ot e_i \ot e^*_j \ot e_t$ is positive is equal to $|i-t|$.
\item[\rm (ii)]
For $0 \leq j \leq d$ the number of $i$ $(0 \leq i \leq d)$ such that
$e^*_s \ot e_i \ot e^*_j \ot e_t$ is negative is equal to $|j-s|$.
\end{itemize}
\end{lemma}

\begin{proof}
Routine using Lemma \ref{lem:ZZ} and Definition \ref{def:positive}.
\end{proof}

\begin{example}  \label{exam:2}
Referring to Example \ref{exam:1},
for each value of $s,t$ we display the positive and negative elements in
$e^*_s \ot D \ot D^* \ot e_t$.
For each table below and for $0 \leq i,j\leq d$,
the $(i,j)$-entry is `$+$' (resp. `$-$') if and only if the element
$e^*_s \ot e_i \ot e^*_j \ot e_t$ is positive (resp. negative).
For $0 \leq i \leq d$ the number of `$+$' in row $i$ is given 
in the column at the right.
For $0 \leq j \leq d$ the number of `$-$' in column $j$ is given 
in the row at the bottom.
\begin{itemize}
\item[(i)]
Assume $s=3$ and $t=6$.
\begin{center}
\begin{tabular}{cc|cccccccccc|c}
  \multicolumn{2}{c}{} & & & & $s$ &    &     & $t$
\\
    &     & $0$ & $1$ & $2$ & $3$ & $4$ & $5$ & $6$ & $7$ & $8$ & $9$ & \#$+$
\\ \cline{2-13}
    & $0$ & $+$ & $+$ & $+$ & $+$ & $+$ & $+$ &  z  &  z  &  z  &  z  & $6$
\\
    & $1$ & $-$ & $+$ & $+$ & $+$ & $+$ & $+$ &  z  &  z  &  z  &  z  & $5$
\\
    & $2$ & $-$ & $-$ & $+$ & $+$ & $+$ & $+$ &  z  &  z  &  z  &  z  & $4$
\\
$s$ & $3$ & $-$ & $-$ & $-$ & $+$ & $+$ & $+$ & $-$ & $-$ & $-$ & $-$ & $3$
\\
    & $4$ &  z  &  z  &  z  & $+$ & $+$ & $-$ & $-$ & $-$ & $-$ & $-$ & $2$
\\
    & $5$ &  z  &  z  &  z  & $+$ & $-$ & $-$ & $-$ & $-$ & $-$ & $-$ & $1$
\\
$t$ & $6$ &  z  &  z  &  z  &  z  &  z  &  z  &  z  & $-$ & $-$ & $-$ & $0$
\\
    & $7$ &  z  &  z  &  z  &  z  &  z  &  z  &  z  & $+$ & $-$ & $-$ & $1$
\\
    & $8$ &  z  &  z  &  z  &  z  &  z  &  z  &  z  & $+$ & $+$ & $-$ & $2$
\\
    & $9$ &  z  &  z  &  z  &  z  &  z  &  z  &  z  & $+$ & $+$ & $+$ & $3$
\\ \cline{2-13}
    & \#$-$ & $3$ & $2$ & $1$ & $0$ & $1$ & $2$ & $3$ & $4$ & $5$ & $6$
\end{tabular}
\end{center}
\item[(ii)]
Assume $s=4$ and $t=4$.
\begin{center}
\begin{tabular}{cc|cccccccccc|c}
 \multicolumn{2}{c}{} & & & &     & $s$
\\
    &     & $0$ & $1$ & $2$ & $3$ & $4$ & $5$ & $6$ & $7$ & $8$ & $9$ & \#$+$
\\ \cline{2-13}
    & $0$ & $+$ & $+$ & $+$ & $+$ &  z  &  z  &  z  &  z  &  z  &  z  & $4$
\\
    & $1$ & $-$ & $+$ & $+$ & $+$ &  z  &  z  &  z  &  z  &  z  &  z  & $3$
\\
    & $2$ & $-$ & $-$ & $+$ & $+$ &  z  &  z  &  z  &  z  &  z  &  z  & $2$
\\
    & $3$ & $-$ & $-$ & $-$ & $+$ &  z  &  z  &  z  &  z  &  z  &  z  & $1$
\\
$s$ & $4$ & $-$ & $-$ & $-$ & $-$ &  z  & $-$ & $-$ & $-$ & $-$ & $-$ & $0$
\\
    & $5$ &  z  &  z  &  z  &  z  &  z  & $+$ & $-$ & $-$ & $-$ & $-$ & $1$
\\
    & $6$ &  z  &  z  &  z  &  z  &  z  & $+$ & $+$ & $-$ & $-$ & $-$ & $2$
\\
    & $7$ &  z  &  z  &  z  &  z  &  z  & $+$ & $+$ & $+$ & $-$ & $-$ & $3$
\\
    & $8$ &  z  &  z  &  z  &  z  &  z  & $+$ & $+$ & $+$ & $+$ & $-$ & $4$
\\
    & $9$ &  z  &  z  &  z  &  z  &  z  & $+$ & $+$ & $+$ & $+$ & $+$ & $5$
\\ \cline{2-13}
    & \#$-$ & $4$ & $3$ & $2$ & $1$ & $0$ & $1$ & $2$ & $3$ & $4$ & $5$
\end{tabular}
\end{center}
\item[(iii)]
Assume $s=6$ and $t=3$.
\begin{center}
\begin{tabular}{cc|cccccccccc|c}
  \multicolumn{2}{c}{}  &     &     &     & $t$ &     &     & $s$
\\
    &     & $0$ & $1$ & $2$ & $3$ & $4$ & $5$ & $6$ & $7$ & $8$ & $9$ & \#$+$
\\ \cline{2-13}
    & $0$ & $+$ & $+$ & $+$ &  z  &  z  &  z  &  z  &  z  &  z  &  z  & $3$
\\
    & $1$ & $-$ & $+$ & $+$ &  z  &  z  &  z  &  z  &  z  &  z  &  z  & $2$
\\
    & $2$ & $-$ & $-$ & $+$ &  z  &  z  &  z  &  z  &  z  &  z  &  z  & $1$
\\
$t$ & $3$ & $-$ & $-$ & $-$ &  z  &  z  &  z  &  z  &  z  &  z  &  z  & $0$
\\
    & $4$ & $-$ & $-$ & $-$ & $-$ & $-$ & $-$ & $+$ &  z  &  z  &  z  & $1$
\\
    & $5$ & $-$ & $-$ & $-$ & $-$ & $-$ & $+$ & $+$ &  z  &  z  &  z  & $2$
\\
$s$ & $6$ & $-$ & $-$ & $-$ & $-$ & $+$ & $+$ & $+$ & $-$ & $-$ & $-$ & $3$
\\
    & $7$ &  z  &  z  &  z  &  z  & $+$ & $+$ & $+$ & $+$ & $-$ & $-$ & $4$
\\
    & $8$ &  z  &  z  &  z  &  z  & $+$ & $+$ & $+$ & $+$ & $+$ & $-$ & $5$
\\
    & $9$ &  z  &  z  &  z  &  z  & $+$ & $+$ & $+$ & $+$ & $+$ & $+$ & $6$
\\ \cline{2-13}
    & \#$-$ & $6$ & $5$ & $4$ & $3$ & $2$ & $1$ & $0$ & $1$ & $2$ & $3$
\end{tabular}
\end{center}
\end{itemize}
\end{example}

\begin{definition}        \label{def:weight}
For each nonzigzag element $e^*_s \ot e_i \ot e^*_j \ot e_t$
we define its {\em weight} as follows.
\begin{itemize}
\item[(i)]
For $s<t$:
\[
\begin{array} {c|c}
 \text{case} &  \text{weight}
\\ \hline \hline
 0 \leq i < s \; \text{ and } \; i \leq j <t  
 & 2i
\\
 s \leq i < t \; \text{ and } \; s \leq j < s+t-i 
 & 2(d-t+i+1)
\\
 t < j \leq i \leq d
 & 2(d-i+s)
\\ \hline
 0 \leq j < i \leq s
 & 2j+1
\\
 s < j \leq t \; \text{ and } \; s+t-j \leq i < t
 & 2(d-j+s)+1
\\
 t < j \leq d \; \text{ and } \; s \leq i < j
 & 2(d-j+s)+1
\end{array}
\]
\item[(ii)]
For $s=t$:
\[
\begin{array} {c|c}
 \text{case} & \text{weight}
\\ \hline \hline
 0 \leq i \leq j < s
 & 2i
\\
 s < j \leq i \leq d
 & 2(d-i+s)
\\ \hline
 0 \leq j < i \leq s
 & 2j+1
\\
 s \leq i < j \leq d
 & 2(d-j+s)+1
\end{array}
\]
\item[(iii)]
For $s>t$:
\[
\begin{array} {c|c}
 \text{case} &  \text{weight}
\\ \hline \hline
 0 \leq i \leq j < t 
 & 2i
\\
 t < i \leq s \; \text{ and } \; s+t-i < j \leq s
 & 2(d-i+t+1)
\\
 s < i \leq d \; \text{ and } \; t < j \leq i
 & 2(d-i+t)
\\ \hline
 0 \leq j < t \; \text{ and } \; j < i \leq s
 & 2j+1
\\
 t \leq j < s \; \text{ and } \;  t < i \leq s+t-j
 & 2(d-s+j)+1
\\
 s \leq i < j \leq d
 & 2(d-j+t)+1
\end{array}
\]
\end{itemize}
\end{definition}

\begin{note}
For $0 \leq i,j\leq d$, 
if $e^*_s \ot e_i \ot e^*_j \ot e_t$ is positive (resp. negative) 
then its weight is even (resp. odd).
\end{note}

\begin{example}  \label{exam:3}
Referring to Example \ref{exam:1},
for each value of $s,t$ we display the weight of the nonzigzag
elements in $e^*_s \ot D \ot D^* \ot e_t$.
For each table below and for $0 \leq i,j\leq d$ such that
$e^*_s \ot e_i \ot e^*_j \ot e_t$ is nonzigzag,
the $(i,j)$-entry is the weight of
$e^*_s \ot e_i \ot e^*_j \ot e_t$.
\begin{itemize}
\item[(i)]
Assume $s=3$ and $t=6$.
\begin{center}
\begin{tabular}{cc|cccccccccc}
  \multicolumn{2}{c}{} & & & & $s$ &    &     & $t$
\\
    &     & $0$ & $1$ & $2$ & $3$ & $4$ & $5$ & $6$ & $7$ & $8$ & $9$
\\ \cline{2-12}
    & $0$ & $0$ & $0$ & $0$ & $0$ & $0$ & $0$ &  z  &  z  &  z  &  z
\\
    & $1$ & $1$ & $2$ & $2$ & $2$ & $2$ & $2$ &  z  &  z  &  z  &  z
\\
    & $2$ & $1$ & $3$ & $4$ & $4$ & $4$ & $4$ &  z  &  z  &  z  &  z
\\
$s$ & $3$ & $1$ & $3$ & $5$ & $14$& $14$& $14$& $13$& $11$& $9$ & $7$
\\
    & $4$ &  z  &  z  &  z  & $16$& $16$& $15$& $13$& $11$& $9$ & $7$
\\
    & $5$ &  z  &  z  &  z  & $18$& $17$& $15$& $13$& $11$& $9$ & $7$
\\
$t$ & $6$ &  z  &  z  &  z  &  z  &  z  &  z  &  z  & $11$& $9$ & $7$
\\
    & $7$ &  z  &  z  &  z  &  z  &  z  &  z  &  z  & $10$& $9$ & $7$
\\
    & $8$ &  z  &  z  &  z  &  z  &  z  &  z  &  z  & $8$ & $8$ & $7$
\\
    & $9$ &  z  &  z  &  z  &  z  &  z  &  z  &  z  & $6$ & $6$ & $6$
\end{tabular}
\end{center}
\item[(ii)]
Assume $s=4$ and $t=4$.
\begin{center}
\begin{tabular}{cc|cccccccccc}
 \multicolumn{2}{c}{} & & & &     & $s$
\\
    &     & $0$ & $1$ & $2$ & $3$ & $4$ & $5$ & $6$ & $7$ & $8$ & $9$
\\ \cline{2-12}
    & $0$ & $0$ & $0$ & $0$ & $0$ &  z  &  z  &  z  &  z  &  z  &  z
\\
    & $1$ & $1$ & $2$ & $2$ & $2$ &  z  &  z  &  z  &  z  &  z  &  z
\\
    & $2$ & $1$ & $3$ & $4$ & $4$ &  z  &  z  &  z  &  z  &  z  &  z 
\\
    & $3$ & $1$ & $3$ & $5$ & $6$ &  z  &  z  &  z  &  z  &  z  &  z 
\\
$s$ & $4$ & $1$ & $3$ & $5$ & $7$ &  z  & $17$& $15$& $13$& $11$& $9$
\\
    & $5$ &  z  &  z  &  z  &  z  &  z  & $16$& $15$& $13$& $11$& $9$
\\
    & $6$ &  z  &  z  &  z  &  z  &  z  & $14$& $14$& $13$& $11$& $9$
\\
    & $7$ &  z  &  z  &  z  &  z  &  z  & $12$& $12$& $12$& $11$& $9$
\\
    & $8$ &  z  &  z  &  z  &  z  &  z  & $10$& $10$& $10$& $10$& $9$
\\
    & $9$ &  z  &  z  &  z  &  z  &  z  & $8$ & $8$ & $8$ & $8$ & $8$
\end{tabular}
\end{center}
\item[(iii)]
Assume $s=6$ and $t=3$.
\begin{center}
\begin{tabular}{cc|cccccccccc}
  \multicolumn{2}{c}{}  &     &     &     & $t$ &     &     & $s$
\\
    &     & $0$ & $1$ & $2$ & $3$ & $4$ & $5$ & $6$ & $7$ & $8$ & $9$
\\ \cline{2-12}
    & $0$ & $0$ & $0$ & $0$ &  z  &  z  &  z  &  z  &  z  &  z  &  z
\\
    & $1$ & $1$ & $2$ & $2$ &  z  &  z  &  z  &  z  &  z  &  z  &  z
\\
    & $2$ & $1$ & $3$ & $4$ &  z  &  z  &  z  &  z  &  z  &  z  &  z
\\
$t$ & $3$ & $1$ & $3$ & $5$ &  z  &  z  &  z  &  z  &  z  &  z  &  z
\\
    & $4$ & $1$ & $3$ & $5$ & $13$& $15$& $17$& $18$&  z  &  z  &  z
\\
    & $5$ & $1$ & $3$ & $5$ & $13$& $15$& $16$& $16$&  z  &  z  &  z
\\
$s$ & $6$ & $1$ & $3$ & $5$ & $13$& $14$& $14$& $14$& $11$& $9$ & $7$
\\
    & $7$ &  z  &  z  &  z  &  z  & $10$& $10$& $10$& $10$& $9$ & $7$
\\
    & $8$ &  z  &  z  &  z  &  z  & $8$ & $8$ & $8$ & $8$ & $8$ & $7$
\\
    & $9$ &  z  &  z  &  z  &  z  & $6$ & $6$ & $6$ & $6$ & $6$ & $6$
\end{tabular}
\end{center}
\end{itemize}
\end{example}

\begin{lemma}         \label{lem:weight}    \samepage
The following {\rm (i)}, {\rm (ii)} hold for $0 \leq i,j,r \leq d$.
\begin{itemize}
\item[\rm (i)]
Assume $e^*_s \ot e_i \ot e^*_j \ot e_t$ is positive and
$e^*_s \ot e_i \ot e^*_r \ot e_t$ is negative.
Then the weight of $e^*_s \ot e_i \ot e^*_j \ot e_t$
is greater than the weight of $e^*_s \ot e_i \ot e^*_r \ot e_t$.
\item[\rm (ii)]
Assume $e^*_s \ot e_i \ot e^*_j \ot e_t$ is negative and
$e^*_s \ot e_r \ot e^*_j \ot e_t$ is positive.
Then the weight of $e^*_s \ot e_i \ot e^*_j \ot e_t$
is greater than the weight of $e^*_s \ot e_r \ot e^*_j \ot e_t$.
\end{itemize}
\end{lemma}

\begin{proof}
Routine using Definitions \ref{def:positive} and \ref{def:weight}.
\end{proof}

\begin{definition}          \label{def:fij}   \samepage
We define the following polynomials in $\K[\lambda]$.
\begin{itemize}
\item[(i)]
For each $i,j$ $(0 \leq i,j \leq d)$ such that
$e^*_s \ot e_i \ot e^*_j \ot e_t$ is positive,
we define a polynomial
\[
  f^+_{ij} = \prod_k \frac{\lambda - \th^*_k}
                          {\th^*_j - \th^*_k},
\]
where the product is over all integers $k$ $(0 \leq k \leq d)$
such that $k \neq j$ and
$e^*_s \ot e_i \ot e^*_k \ot e_t$ is positive.
The polynomial $f^+_{ij}$ has degree $|i-t|-1$ by Lemma \ref{lem:i-t}(i).
\item[(ii)]
For each $i,j$ $(0 \leq i,j \leq d)$ such that
$e^*_s \ot e_i \ot e^*_j \ot e_t$ is negative,
we define a polynomial
\[
  f^-_{ij} = \prod_k \frac{\lambda - \th_k}
                          {\th_i - \th_k},
\]
where the product is over all integers $k$ $(0 \leq k \leq d)$
such that $k \neq i$ and 
$e^*_s \ot e_k \ot e^*_j \ot e_t$ is negative.
The polynomial $f^-_{ij}$ has degree $|j-s|-1$ by Lemma \ref{lem:i-t}(ii).
\end{itemize}
\end{definition}

\begin{lemma}     \label{lem:fij}
For $0 \leq i,j \leq d$ the following {\rm (i)}, {\rm (ii)} hold.
\begin{itemize}
\item[\rm (i)]
Assume $e^*_s \ot e_i \ot e^*_j \ot e_t$ is positive.
Then 
\[
  e^*_s \ot e_i \ot e^*_j \ot e_t =
      e^*_s \ot e_i \ot f^+_{ij}(a^*) \ot e_t 
    - \sum_h f^+_{ij}(\th^*_h) e^*_s \ot e_i \ot e^*_h \ot e_t,
\]
where the sum is over all integers $h$ $(0 \leq h \leq d)$
such that 
$e^*_s \ot e_i \ot e^*_h \ot e_t$ is zigzag or negative.
\item[\rm (ii)]
Assume $e^*_s \ot e_i \ot e^*_j \ot e_t$ is negative.
Then 
\[
  e^*_s \ot e_i \ot e^*_j \ot e_t =
      e^*_s \ot f^-_{ij}(a) \ot e^*_j \ot e_t 
    - \sum_h f^-_{ij}(\th_h) e^*_s \ot e_h \ot e^*_j \ot e_t,
\]
where the sum is over all integers $h$ $(0 \leq h \leq d)$
such that 
$e^*_s \ot e_h \ot e^*_j \ot e_t$ is zigzag or positive.
\end{itemize}
\end{lemma}

\begin{proof}
(i):
Using $a^* = \sum_{h=0}^d \th^*_h e^*_h$ we find
\begin{equation}     \label{eq:fijas}
  f^+_{ij}(a^*) = \sum_{h=0}^d f^+_{ij}(\th^*_h)e^*_h.
\end{equation}
We evaluate the right side of \eqref{eq:fijas}.
By Definition \ref{def:fij}(i) we have $f^+_{ij}(\th^*_j)=1$.
We have also
$f^+_{ij}(\th^*_h)=0$ if $h \neq j$ and
$e^*_s \ot e_i \ot e^*_h \ot e_t$ is positive.
Thus the right side of \eqref{eq:fijas} is equal to
$e^*_j + \sum_{h} f^+_{ij}(\th^*_h)e^*_h$,
where the sum is over all integers $h$ $(0 \leq h \leq d)$
such that $e^*_s \ot e_i \ot e^*_h \ot e_t$ is zigzag or negative.
The result follows.

(ii): Similar to the proof of (i).
\end{proof}

\begin{theorem}     \label{thm:mainn4}    \samepage
For $0 \leq s,t \leq d$
the following elements {\rm (i)}--{\rm (iii)} together form a basis 
for $e^*_s \ot D \ot D^* \ot e_t$:
\begin{itemize}
\item[\rm (i)] 
the elements \eqref{eq:Rstright};
\item[\rm (ii)]
the elements \eqref{eq:Rstleft};
\item[\rm (iii)]
the zigzag elements in $e^*_s \ot D \ot D^* \ot e_t$.
\end{itemize}
\end{theorem}

\begin{proof}
The dimension of $e^*_s \ot D \ot D^* \ot e_t$ is $(d+1)^2$. 
The number of elements in (i)--(iii) are given in 
Lemmas \ref{lem:sizeRst} and \ref{lem:sizeZZ}, 
and the sum of these numbers is $(d+1)^2$.
Therefore the number of elements in (i)--(iii) is equal to
the dimension of $e^*_s \ot D \ot D^* \ot e_t$.
To finish the proof, it suffices to show that (i)--(iii) together span 
$e^*_s \ot D \ot D^* \ot e_t$.
Let $S$ denote the subspace of $e^*_s \ot D \ot D^* \ot e_t$ 
spanned by (i)--(iii).
We show $S=e^*_s \ot D \ot D^* \ot e_t$. 
To this end we fix $i,j$ $(0 \leq i,j\leq d)$ and show
\[
    e^*_s \ot e_i \ot e^*_j \ot e_t  \in S.
\]
If $e^*_s \ot e_i \ot e^*_j \ot e_t$ is zigzag
then it is contained in $S$ by the definition of $S$.
So we may assume
$e^*_s \ot e_i \ot e^*_j \ot e_t$ is nonzigzag.
Let $w$ denote its weight.
By induction on the weight, we may assume that
each nonzigzag element of
   $e^*_s \ot D \ot D^* \ot e_t$
that has weight less than $w$ is contained in $S$.
First assume $e^*_s \ot e_i \ot e^*_j \ot e_t$ is positive.
Consider the equation in Lemma \ref{lem:fij}(i).
On the right hand side the first term is in $R_{s,t}$,
so it is contained in $S$.
On the right hand side the remaining terms are either
zigzag or negative.
The negative terms have weight less than $w$ in view of 
Lemma \ref{lem:weight}(i).
In either case they are in $S$, and therefore
$e^*_s \ot e_i \ot e^*_j \ot e_t$ is in $S$.
Next assume
$e^*_s \ot e_i \ot e^*_j \ot e_t$ is negative.
Consider the equation in Lemma \ref{lem:fij}(ii).
On the right hand side the first term is in $R_{s,t}$,
so it is contained in $S$.
On the right hand side the remaining terms are either
zigzag or positive.
The positive terms have weight less than $w$ in view of 
Lemma \ref{lem:weight}(ii).
In either case they are in $S$, and therefore
$e^*_s \ot e_i \ot e^*_j \ot e_t$ is in $S$.
\end{proof}

\begin{corollary}      \samepage
For $0 \leq s,t \leq d$
the elements \eqref{eq:Rstright} and \eqref{eq:Rstleft} together form a basis
for $R_{s,t}$.
\end{corollary}

\begin{corollary}    \label{cor:main4}   \samepage
For $0 \leq s,t \leq d$ 
the zigzag elements in $e^*_s \ot D \ot D^* \ot e_t$ form a basis for
a complement of $R_{s,t}$ in $e^*_s \ot D \ot D^* \ot e_t$.
\end{corollary}

\begin{corollary}    \samepage
The following elements {\rm (i)}--{\rm (iii)} together form a basis 
for $D^* \ot D \ot D^* \ot D$:
\begin{itemize}
\item[\rm (i)]
the elements \eqref{eq:Rright};
\item[\rm (ii)]
the elements \eqref{eq:Rleft};
\item[\rm (iii)]
the zigzag elements in $D^* \ot D \ot D^* \ot D$.
\end{itemize}
\end{corollary}

\begin{corollary}    \samepage
The elements \eqref{eq:Rright} and \eqref{eq:Rleft} together form a basis
for $R$.
\end{corollary}

\begin{corollary}    \samepage
The zigzag elements in $D^* \ot D \ot D^* \ot D$ form a basis for
a complement of $R$ in $D^* \ot D \ot D^* \ot D$.
\end{corollary}

\begin{corollary}    \samepage
The following {\rm (i)}--{\rm (iv)} hold.
\begin{itemize}
\item[\rm (i)]
For $0 \leq s,t \leq d$ the dimension of $R_{s,t}$ is 
$d(d+1)-(s+t)d+s^2+t^2$.
\item[\rm (ii)]
For $0 \leq s,t \leq d$ the codimension of $R_{s,t}$ in 
$e^*_s \ot D \ot D^* \ot e_t$ is $d+1+s(d-s)+t(d-t)$.
\item[\rm (iii)]
The dimension of $R$ is $2d(d+1)^2(d+2)/3$.
\item[\rm (iv)]
The codimension of $R$ in $D^* \ot D \ot D^* \ot D$ is 
$(d+1)^2(d^2+2d+3)/3$.
\end{itemize}
\end{corollary}

\medskip

We are done with our analysis of $D^* \ot D \ot D^* \ot D$.
In the next section we return to $T$.

\section{The zigzag words in $T$}

\indent
From now until the end of Section \ref{sec:proofofmain} the following 
notation will be in effect. 
Fix a nonnegative integer $d$. 
Fix a sequence $(\{\th_i\}_{i=0}^d; \{\th^*_i\}_{i=0}^d)$ of scalars 
taken from $\K$ that satisfies \eqref{eq:distinct} and Lemma \ref{lem:corec}.
Let $T$ denote the corresponding $\K$-algebra from Definition \ref{def:T}.
We will be discussing some special elements of $T$.
To facilitate this discussion we make some definitions.
Recall the idempotent generators $\{e_i\}_{i=0}^d$, $\{e^*_i\}_{i=0}^d$ for $T$.
We will call the $\{e^*_i\}_{i=0}^d$ {\em starred} and the $\{e_i\}_{i=0}^d$
{\em nonstarred}.
A pair of idempotent generators for $T$ will be called {\em alternating} 
whenever one of them is starred and the other is nonstarred.

\medskip

\begin{definition}   \label{def:word}  \samepage
For an integer $n\geq 0$, by a {\em word of length $n$ in $T$} we mean 
a product $x_1x_2 \cdots x_n$ such that $\{x_i\}_{i=1}^n$ are idempotent 
generators for $T$ and $x_{i-1}, x_i$ are alternating for $2 \leq i \leq n$. 
We interpret the word of length $0$ to be the identity of $T$. 
We call this word {\em trivial}. 
Let $x_1x_2\cdots x_n$ denote a nontrivial word in $T$. 
We say that this word {\em begins} with $x_1$ and {\em ends} with $x_n$.
\end{definition}

\begin{example}
For $d=2$ we display the words in $T$ that have length $3$ and
begin with $e_0$.
\begin{align*}
  & e_0e^*_0e_0, \qquad e_0e^*_0e_1, \qquad e_0e^*_0e_2,  \\
  & e_0e^*_1e_0, \qquad e_0e^*_1e_1, \qquad e_0e^*_1e_2,  \\
  & e_0e^*_2e_0, \qquad e_0e^*_2e_1, \qquad e_0e^*_2e_2.
\end{align*}
\end{example}

\medskip

Referring to Definition \ref{def:word}, observe that $T$ is spanned by
its words. We now define a special type of word said to be zigzag.

\medskip

\begin{definition}  \samepage
For an idempotent generator $e_i$ or $e^*_i$ we call $i$ the {\em index} 
of the generator.
For an idempotent generator $x$ let $\ox$ denote the index of $x$.
\end{definition}

\begin{definition}   \label{def:zz}   \samepage
A word $x_1x_2\cdots x_n$ in $T$ is said to be 
{\em zigzag} whenever both
\begin{itemize}
\item[(i)]
$\ox_i$ is not between $\ox_{i-1}$, $\ox_{i+1}$ for $2 \leq i \leq n-1$;
\item[(ii)]
at least one of $\ox_{i-1}$, $\ox_{i}$ is not between $\ox_{i-2}$, 
$\ox_{i+1}$ for $3 \leq i \leq n-1$.
\end{itemize}
\end{definition}

\medskip

In Definition \ref{def:zz} we defined the zigzag words in $T$.
In Definitions \ref{def:zigzagelement1}, \ref{def:zigzagelement} 
some other versions of zigzag were given.
We now compare Definitions \ref{def:zigzagelement1}, \ref{def:zigzagelement} and 
Definition \ref{def:zz}.

\medskip

\begin{lemma}    \label{lem:equiv1}   \samepage
For $0 \leq s,i,t \leq d$ the following {\rm (i)}, {\rm (ii)} are equivalent.
\begin{itemize}
\item[\rm (i)]
$e^*_s \otimes e_i \otimes e^*_t$ is zigzag in the sense of 
Definition {\rm \ref{def:zigzagelement1}}.
\item[\rm (ii)] 
The word $e^*_se_ie^*_t$ is zigzag in the sense of Definition 
{\rm \ref{def:zz}}.
\end{itemize}
\end{lemma}

\begin{lemma}    \label{lem:equiv}   \samepage
For $0 \leq s,i,j,t \leq d$ the following {\rm (i)}, {\rm (ii)} are equivalent.
\begin{itemize}
\item[\rm (i)]
$e^*_s \otimes e_i \otimes e^*_j \otimes e_t$ is zigzag in the sense of 
Definition {\rm \ref{def:zigzagelement}}.
\item[\rm (ii)] 
The word $e^*_se_ie^*_je_t$ is zigzag in the sense of Definition 
{\rm \ref{def:zz}}.
\end{itemize}
\end{lemma}

\medskip

We now describe the zigzag words in $T$ from several points of view.
In this description we will use the following notion. 
Two integers $a,b$ are said to have {\em opposite sign}
whenever $ab \leq 0$.

\medskip

\begin{theorem}       \label{thm:zz}   \samepage
Let $x_1x_2\cdots x_n$ denote a word in $T$.
Then this word is zigzag if and only if both
\begin{itemize}
\item[\rm (i)]
$\ox_{i-1}-\ox_{i}$ and $\ox_{i}-\ox_{i+1}$ have opposite sign
for $2 \leq i \leq n-1$;
\item[\rm (ii)]
for $2 \leq i \leq n-1$, if $\;|\ox_{i-1}-\ox_{i}|<|\ox_{i}-\ox_{i+1}|$
then
\[
 0 < |\ox_1-\ox_2| < |\ox_2-\ox_3| < \cdots < |\ox_{i}-\ox_{i+1}|.
\]
\end{itemize}
\end{theorem}

\begin{proof}
First assume $x_1x_2\cdots x_n$ is zigzag.
The assertion (i) follows from Definition \ref{def:zz}(i).
Concerning (ii), note by Definition \ref{def:zz}(i) that for $1 \leq j \leq n-1$,
$\ox_j =\ox_{j+1}$ implies $\ox_j=\ox_{j+1}=\cdots =\ox_{n}$.
Note by Definition \ref{def:zz}(i),(ii) that for $3 \leq j \leq n-1$,
$|\ox_{j-1}-\ox_j| < |\ox_{j}-\ox_{j+1}|$ implies
$|\ox_{j-2}-\ox_{j-1}| < |\ox_{j-1}-\ox_{j}|$.
Assertion (ii) follows from these comments.

Next assume (i), (ii) hold. 
To show that $x_1x_2 \cdots x_n$ is zigzag, we show that 
$x_1x_2\cdots x_n$ satisfies the two conditions in Definition \ref{def:zz}.
Concerning Definition \ref{def:zz}(i), pick an integer $i$ 
$(2 \leq i \leq n-1)$, and suppose $\ox_i$ is between $\ox_{i-1}$, $\ox_{i+1}$.
By Definition \ref{def:between}, either 
$\ox_{i-1}\geq \ox_i > \ox_{i+1}$ or $\ox_{i-1}\leq \ox_i < \ox_{i+1}$.
Observe that
$\ox_{i-1}= \ox_i > \ox_{i+1}$ contradicts (ii),
$\ox_{i-1}> \ox_i > \ox_{i+1}$ contradicts (i),
$\ox_{i-1}= \ox_i < \ox_{i+1}$ contradicts (ii), and
$\ox_{i-1}< \ox_i < \ox_{i+1}$ contradicts (i).
In any case there is a contradiction, 
so $\ox_i$ is not between $\ox_{i-1}, \ox_{i+1}$.
We have shown Definition \ref{def:zz}(i) is satisfied.
Concerning Definition \ref{def:zz}(ii), pick an integer $i$ 
$(3 \leq i \leq n-1)$, and suppose each of $\ox_{i-1}$, $\ox_i$ is between 
$\ox_{i-2}$, $\ox_{i+1}$.
By Definition \ref{def:between} and since $\ox_{i-1}$ is between 
$\ox_{i-2}$, $\ox_{i+1}$ we find $\ox_{i-2} \geq \ox_{i-1} > \ox_{i+1}$ or
$\ox_{i-2} \leq \ox_{i-1} < \ox_{i+1}$.
First assume $\ox_{i-2} \geq \ox_{i-1} > \ox_{i+1}$.
Using $\ox_{i-2} > \ox_{i+1}$ and the fact that $\ox_i$ is between 
$\ox_{i-2}$, $\ox_{i+1}$ we find $\ox_{i-2}\geq \ox_i > \ox_{i+1}$.
Since $\ox_i - \ox_{i+1} > 0$ and since $\ox_{i-1} - \ox_i$, $\ox_i-\ox_{i+1}$
have opposite sign, we find $\ox_{i-1}-\ox_i \leq 0$.
By this and since $\ox_{i-1} > \ox_{i+1}$ we find 
$|\ox_{i-1}-\ox_i| < |\ox_i-\ox_{i+1}|$. 
Now using assumption (ii) we obtain 
$|\ox_{i-2}-\ox_{i-1}| < |\ox_{i-1}-\ox_i|$.
But this implies $\ox_{i-2} < \ox_i$ which contradicts our above remarks.
Therefore at least one of $\ox_{i-1}$, $\ox_i$ is not between 
$\ox_{i-2}$, $\ox_{i+1}$.
We have shown Definition \ref{def:zz}(ii) is satisfied for the case
$\ox_{i-2} \geq \ox_{i-1} > \ox_{i+1}$.
The proof of Definition \ref{def:zz}(ii) for the case
$\ox_{i-2} \leq \ox_{i-1} < \ox_{i+1}$ is similar,
and omitted.
\end{proof}

\begin{definition}    \label{def:pq}  
A word $x_1x_2\cdots x_n$ in $T$ is said to be {\em constant} whenever the index
$\ox_i$ is independent of $i$ for $1 \leq i \leq n$.
Note that the trivial word is constant, and each constant word is zigzag.
\end{definition}

\begin{theorem}   \label{thm:p}  \samepage
Let $x_1x_2\cdots x_n$ denote a nonconstant zigzag word.
Then there exists a unique integer $p$ $(2 \leq p \leq n)$ such that both
\begin{itemize}
\item[\rm (i)]
 $0 < |\ox_1-\ox_2| < \cdots < |\ox_{p-1}-\ox_p|$;
\item[\rm (ii)]
 $|\ox_{p-1}-\ox_p| \geq |\ox_p-\ox_{p+1}| \geq \cdots \geq |\ox_{n-1}-\ox_n|$.
\end{itemize}
\end{theorem}

\begin{proof}
Note that $n \geq 2$ since our word is nonconstant.
Note that $\ox_1 \not=\ox_2$;
otherwise $\ox_1=\ox_2 = \cdots = \ox_n$ by Definition \ref{def:zz}(i),
a contradiction. Therefore $0< |\ox_1-\ox_2|$.
Now define $p$ to be the maximal integer $i$ $(2 \leq i \leq n)$ such
that $0 < |\ox_1-\ox_2| < \cdots < |\ox_{i-1}-\ox_i|$.
Then (i) holds, and (ii) follows from Theorem \ref{thm:zz}(ii).
It is clear that $p$ is unique.
\end{proof}

\begin{example}
In the table below we list some nonconstant zigzag words of length $7$.
For each word we display the parameter $p$ from Theorem \ref{thm:p}.
\[
\begin{array}{cc|cc}
   \text{word} & & & p  \\
\hline
 e_{0}e^*_{1}e_{1}e^*_{1}e_{1}e^*_{1}e_{1} & & & 2   \\
 e_{0}e^*_{1}e_{0}e^*_{0}e_{0}e^*_{0}e_{0} & & & 2   \\
 e_{3}e^*_{1}e_{3}e^*_{2}e_{2}e^*_{2}e_{2} & & & 2   \\
 e_{0}e^*_{2}e_{1}e^*_{2}e_{1}e^*_{2}e_{1} & & & 2   \\
 e_{2}e^*_{3}e_{1}e^*_{1}e_{1}e^*_{1}e_{1} & & & 3   \\
 e_{4}e^*_{1}e_{5}e^*_{1}e_{2}e^*_{2}e_{2} & & & 3   \\
 e_{3}e^*_{4}e_{1}e^*_{5}e_{0}e^*_{6}e_{2} & & & 6   \\
 e_{3}e^*_{2}e_{4}e^*_{1}e_{5}e^*_{0}e_{0} & & & 6   \\ 
 e_{4}e^*_{3}e_{5}e^*_{2}e_{7}e^*_{0}e_{8} & & & 7
\end{array}
\]
\end{example}

\section{Reducing to zigzag words}
\label{sec:zzspan}

\indent
In this section we prove the following theorem.

\medskip

\begin{theorem}  \label{thm:span}  \samepage
For an integer $n\geq 1$ and idempotent generators $x,y$ for $T$, 
the following sets have the same span.
\begin{itemize}
\item[\rm (i)] 
The words of length $n$ in $T$, that begin with $x$ and end with $y$.
\item[\rm (ii)] 
The zigzag words of length $n$ in $T$, that begin with $x$ and end with $y$.
\end{itemize}
\end{theorem}

\begin{proof}
Without loss of generality we assume that $x$ is starred.
Note that $y$ is nonstarred if $n$ is even and starred if $n$ is odd.
Let $s$ (resp. $t$) denote the index of $x$ (resp. $y$).
Replacing $e_i$, $e^*_i$ by $e_{d-i}$, $e^*_{d-i}$ $(0 \leq i \leq d)$ 
if necessary, we may assume $s\leq t$.
Let $W$ (resp. $Z$) denote the set of words (resp. zigzag words) of length $n$ 
in $T$ that begin with $x$ and end with $y$.
By construction $W \supseteq Z$.
Our goal is to show that $W$ and $Z$ have the same span.
We will do this by induction on $n$.

The result holds for $n\leq 2$, since in this case $W=Z$. 
For $n=3$ the result is a routine consequence of 
Lemma \ref{lem:piR1}, Corollary \ref{cor:zigzagDDD}, and Lemma \ref{lem:equiv1}.
Next assume $n=4$.
By Corollary \ref{cor:main4} the set of zigzag elements in 
$e^*_s \ot D \ot D^* \ot e_t$ form a basis for a complement of $R_{s,t}$ in
$e^*_s \ot D \ot D^* \ot e_t$.
Let $\pi$ denote the map from Lemma \ref{lem:piR}.
By Lemma \ref{lem:piR} and Lemma \ref{lem:sumRst} the space $R_{s,t}$ is in 
the kernel of $\pi$. 
By construction $\pi$ sends $e^*_s \otimes D \otimes D^* \otimes e_t$
to the span of $W$.
By Lemma \ref{lem:equiv} the map $\pi$ sends the set of zigzag elements in 
$e^*_s \ot D \ot D^* \ot e_t$ to the set $Z$.
By these comments $W$ and $Z$ have the same span.

Next assume $n\geq 5$. 
By way of contradiction assume that $W$ and $Z$ do not have the same span, 
so $W \setminus \text{Span}(Z)$ is nonempty. 
Let $P$ denote the set of ordered pairs $(k,r)$ such that $2 \leq k \leq n-1$ 
and $0 \leq r \leq d$.
Pick any $(k,r) \in P$. By a {\em witness} for $(k,r)$ we mean a word
$x_1x_2 \cdots x_n \in W \setminus \text{Span}(Z)$ such that $\ox_k=r$.
Let $W(k,r)$ denote the set of witnesses for $(k,r)$.
We call $(k,r)$ {\em feasible} whenever $W(k,r)$ is nonempty.
Observe that for all integers $k$ $(2 \leq k \leq n-1)$ there exists an 
integer $r$ $(0 \leq r \leq d)$ such that $(k,r)$ is feasible.

\medskip

{\em Claim 1}.
{\em
For all feasible $(k,r) \in P$ there exists a word
$x_1x_2\cdots x_n \in W(k,r)$ such that each of $x_1x_2\cdots x_k$ and 
$x_kx_{k+1}\cdots x_n$ is zigzag.
}

\medskip

{\em Proof of Claim 1.}
Since $(k,r)$ is feasible there exists $y_1y_2 \cdots y_n \in W(k,r)$. 
By construction $y_1=x$, $y_n=y$, and $\oy_k=r$.
By induction $y_1y_2\cdots y_k$ is contained in the span of the zigzag words 
of length $k$ that begin with $x$ and end with $y_k$. 
Also by induction $y_ky_{k+1}\cdots y_n$ is contained in the span of the 
zigzag words of length $n-k+1$ that begin with $y_k$ and end with $y$.
It follows that there exists a word $x_1x_2\cdots x_n \in W(k,r)$ such that 
each of $x_1x_2\cdots x_k$ and $x_kx_{k+1}\cdots x_n$ is zigzag.
This proves Claim 1.

\medskip

{\em Claim 2}.
{\em
For all feasible $(k,r) \in P$ we have $r \geq s$.
}

\medskip

{\em Proof of Claim 2}.
Suppose the claim is false. 
Then there exists a feasible $(k,r) \in P$ such that $r<s$.
Without loss we may choose $(k,r)$ such that for all $(k',r') \in P$, 
$(k',r')$ is not feasible if (i) $r'<r$, or (ii) $r'=r$ and $k'<k$.
By Claim 1 there exists $x_1x_2\cdots x_n \in W(k,r)$ such that each of 
$x_1x_2\cdots x_k$ and $x_kx_{k+1}\cdots x_n$ is zigzag. 
We will show that $x_1x_2\cdots x_n$ is zigzag.
By construction $\ox_k=r$.
Also by construction $\ox_i>r$ for $1 \leq i \leq k-1$ and 
$\ox_i \geq r$ for $k+1 \leq i \leq n$.
Therefore $\ox_k$ is not between $\ox_i$, $\ox_j$ for $1 \leq i < k < j \leq n$.
We now show that $x_1x_2\cdots x_n$ is zigzag by showing that it
satisfies the two conditions of Definition \ref{def:zz}.
Concerning Definition \ref{def:zz}(i), 
for $2 \leq i \leq n-1$ we must show that
$\ox_i$ is not between $\ox_{i-1}$ and $\ox_{i+1}$. 
This is the case for $i \leq k-1$ since $x_1x_2\cdots x_k$ is zigzag,
it is the case for $i=k$ by our above comments,
and it is the case for $i \geq k+1$ since $x_kx_{k+1} \cdots x_n$ is zigzag.
Concerning Definition \ref{def:zz}(ii), 
for $3 \leq i \leq n-1$ we must show that at least one of 
$\ox_{i-1}$, $\ox_i$ is not between $\ox_{i-2}$ and $\ox_{i+1}$.
This is the case for $i \leq k-1$ since $x_1x_2\cdots x_k$ is zigzag,
it is the case for $k\leq i \leq k+1$ by our above comments,
and it is the case for  $i \geq k+2$ since $x_kx_{k+1} \cdots x_n$ is zigzag.
We have shown $x_1x_2\cdots x_n$ satisfies the conditions of 
Definition \ref{def:zz}.
Therefore $x_1x_2\cdots x_n$ is zigzag.
This is a contradiction since $x_1x_2\cdots x_n \in W(k,r)$.
This proves Claim 2.

\medskip
{\em Claim 3}.
{\em
For all feasible $(k,r) \in P$ we have $r < t$.
}

\medskip

{\em Proof of Claim 3}.
Suppose the claim is false.
Then there exists a feasible $(k,r) \in P$ such that $r \geq t$.
Without loss we may choose $(k,r)$ such that for all $(k',r') \in P$, 
$(k',r')$ is not feasible if (i) $r'>r$, or (ii) $r'=r$ and $k'<k$.
By Claim 1 there exists $x_1x_2\cdots x_n \in W(k,r)$ such that each of 
$x_1x_2\cdots x_k$ and $x_kx_{k+1}\cdots x_n$ is zigzag. 
We will show that $x_1x_2\cdots x_n$ is zigzag.
Suppose for the moment that $r=s$.
In this case $s=t$ since we assume $s\leq t \leq r$.
Now by the construction and Claim 2 the word $x_1x_2\cdots x_n$ is constant 
and hence zigzag.
Next assume $r>s$.
By construction $\ox_k=r$.
Also by construction $\ox_i<r$ for $1 \leq i \leq k-1$ and 
$\ox_i \leq r$ for $k+1 \leq i \leq n$.
Therefore $\ox_k$ is not between $\ox_i$, $\ox_j$ for $1 \leq i < k < j \leq n$.
We now show that $x_1x_2\cdots x_n$ is zigzag, by showing that it satisfies 
the two conditions of Definition \ref{def:zz}.
Concerning Definition \ref{def:zz}(i), 
for $2 \leq i \leq n-1$ we must show that $\ox_i$ is not between 
$\ox_{i-1}$ and $\ox_{i+1}$. 
This is the case for $i \leq k-1$ since $x_1x_2\cdots x_k$ is zigzag,
it is the case for $i=k$ by our above comments,
and it is the case for $i \geq k+1$ since $x_kx_{k+1} \cdots x_n$ is zigzag.
Concerning Definition \ref{def:zz}(ii), 
for $3 \leq i \leq n-1$ we must show that at least one of 
$\ox_{i-1}$, $\ox_i$ is not between $\ox_{i-2}$, $\ox_{i+1}$.
This is the case for $i \leq k-1$ since $x_1x_2\cdots x_r$ is zigzag,
it is the case for $k\leq i \leq k+1$ by our above comments,
and it is the case for  $i \geq k+2$ since $x_kx_{k+1} \cdots x_n$ is zigzag.
We have shown $x_1x_2\cdots x_n$ satisfies the conditions of 
Definition \ref{def:zz}.
Therefore $x_1x_2\cdots x_n$ is zigzag.
This is a contradiction since $x_1x_2\cdots x_n \in W(k,r)$.
This proves Claim 3.

\medskip

Let $t'$ denote the minimal integer $r$ $(0 \leq r \leq d)$ such that 
$(n-1,r)$ is feasible.
Observe that $s\leq t'$ by Claim 2 and $t' < t$ by Claim 3.
Let $P'$ denote the set of ordered pairs $(k,r)$ such that $2 \leq k \leq n-2$ 
and $0 \leq r\leq d$.

\medskip

{\em Claim 4}.
{\em
For all feasible $(k,r) \in P'$ we have $r>t'$.
}

\medskip

{\em Proof of Claim 4}.
By Claim 1 and since $P' \subseteq P$, there exists a word 
$x_1x_2\cdots x_n \in W(k,r)$ such that $x_kx_{k+1}\cdots x_n$ is zigzag.
By construction $\ox_k = r$ and $\ox_n = t$.
By Claim 3 $\ox_{n-2} < t$ and $\ox_{n-1} <  t$.
Also $\ox_{n-1} \geq  t'$ by the definition of $t'$.
Since $x_kx_{k+1}\cdots x_n$ is zigzag and $k\leq n-2$, 
the integer $\ox_{n-1}$ is not between $\ox_{n-2}$ and $\ox_{n}$.
By this and since $\ox_{n-1} < \ox_n$ we find $\ox_{n-2} > \ox_{n-1}$.
Now $\ox_{n-1} < \ox_{n-2} < \ox_n$ so 
$|\ox_{n-2} - \ox_{n-1}| < |\ox_{n-1} - \ox_n|$.
With this in mind we apply Theorem \ref{thm:zz} to $x_kx_{k+1}\cdots x_n$
and get $\ox_k > \ox_{n-1}$.
Recall $\ox_k =r$ and $\ox_{n-1} \geq t'$ so $r>t'$.
This proves Claim 4.

\medskip

By construction the pair $(n-1,t')$ is a feasible element of $P$. 
Now by Claim 1 there exists $x_1x_2\cdots x_n \in W(n-1,t')$ such that
$x_1x_2\cdots x_{n-1}$ is zigzag.
By construction $\ox_1 = s$ and $\ox_{n-1} = t'$.
By Claim 4 and since $n\geq 5$ we have $\ox_{n-2} > \ox_{n-1}$ and
$\ox_{n-3} > \ox_{n-1}$.
Since $x_1x_2\cdots x_{n-1}$ is zigzag, the integer $\ox_{n-2}$ is not between
$\ox_{n-3}$ and $\ox_{n-1}$.
By this and since $\ox_{n-2} > \ox_{n-1}$ we find $\ox_{n-3} < \ox_{n-2}$.
Therefore $\ox_{n-1} < \ox_{n-3} < \ox_{n-2}$ so 
$|\ox_{n-3} - \ox_{n-2}| < |\ox_{n-2} - \ox_{n-1}|$.
With this in mind we apply Theorem \ref{thm:zz} to $x_1x_2\cdots x_{n-1}$
and get $\ox_1 > \ox_{n-1}$.
Recall $\ox_1 = s$ and $\ox_{n-1}=t'$ so $s> t'$. 
This contradicts our earlier statement that $s \leq t'$.
We conclude that $W$ and $Z$ have the same span, and the theorem is proved.
\end{proof}

\section{The proof of Theorem \ref{thm:main}}
\label{sec:proofofmain}

\indent
In this section we prove Theorem \ref{thm:main}. We will use
the following concepts.

\begin{definition}
A word in $T$ will be called {\em lifting} whenever it is nontrivial, 
zigzag, and ends with $e_0$ or $e^*_0$.
Let  $x_1 x_2 \cdots x_n$ denote a lifting word in $T$. 
This word will be called {\em redundant} whenever there exists an integer 
$i$ $(1 \leq i \leq n-1)$ such that $\ox_i=0$.
We call this word {\em nonredundant} whenever it is not redundant.
\end{definition}

\begin{example}
For $d=3$ we display the nonredundant lifting words in $T$ that begin with 
a nonstarred element.
\[
 \begin{array}{ccc}
  & e_0 \\
  e_1e^*_0 & e_1e^*_2e_0 & e_1e^*_3e_0  \\
  e_2e^*_0 & e_2e^*_3e_0 & e_2e^*_1e_3e^*_0  \\
  & e_3e^*_0
 \end{array}
\]
\end{example}

\begin{lemma}    \label{lem:ds}   \samepage
For $0 \leq s \leq d$ there are exactly $\binom{d}{s}$ nonredundant lifting 
words in $T$ that begin with $e_s$.
\end{lemma}

\begin{proof}
Let $L$ denote the set of all nonredundant lifting words in $T$ 
that begin with a nonstarred element. 
Abbreviate $\Omega$ for the set $\{1,2,\ldots, d\}$,
and let $2^\Omega$ denote the set of all subsets of $\Omega$.
We will display a bijection $f:L \to 2^\Omega$ with the property
that for all $x_1x_2\cdots x_n \in L$ the image under $f$ has cardinality $\ox_1$.
To this end pick any $w=x_1 x_2 \cdots x_n \in L$. 
We will define $f(w)$ after a few comments.
Using Theorem \ref{thm:p} we find
\[
  0 < \ox_{n-2} < \ox_{n-4} < \cdots < \ox_4 < \ox_2 
     < \ox_1 < \ox_3 < \cdots < \ox_{n-3} < \ox_{n-1}
\]
if $n$ is even, and
\[
  0 < \ox_{n-2} < \ox_{n-4} < \cdots < \ox_3 < \ox_1 
      < \ox_2 < \ox_4 < \cdots < \ox_{n-3} < \ox_{n-1}
\]
if $n$ is odd.
Define
\begin{align*}
  X^+ &=  \{\ox_i \,|\, 2 \leq i \leq n-1,\; \text{ $n-i\;$ odd}\},   \\
  X^- &=  \{\ox_i \,|\, 1 \leq i \leq n-2,\; \text{ $n-i\;$ even}\}.
\end{align*}
Observe that
\begin{align*}
  X^+ &\subseteq \{\ox_1+1, \ox_1+2, \ldots, d\},  \\
  X^- &\subseteq \{1,2,\ldots, \ox_1\}.
\end{align*}
Further observe that each of $|X^+|$, $|X^-|$ is equal to
$(n-2)/2$ if $n$ is even and $(n-1)/2$ if n is odd.
Therefore $|X^+|=|X^-|$.
We define
\[
  f(w) = X^+ \cup (\{1,2,\ldots,\ox_1\} \setminus X^-).
\]
By the above comments $f(w)$ is a subset of $\Omega$ with cardinality $\ox_1$.
One checks that the map $f:L \to 2^\Omega$ is bijective and the result follows.
\end{proof}

\begin{proposition}     \label{prop:esTes0}  \samepage
For $0 \leq s \leq d$ we have
\begin{equation}    \label{eq:aux}
  e_sTe^*_0 = \sum_{w} w e_0 T e^*_0  + \sum_{w'} w' e^*_0 T e^*_0,
\end{equation}
where the first sum (resp. second sum) is over all the nonredundant lifting 
words $w$ (resp. $w'$) in $T$ that begin with $e_s$ and end with $e_0$ 
(resp. end with $e^*_0$).
\end{proposition}

\begin{proof}
In \eqref{eq:aux} the right-hand side is contained in the left-hand side
since each summand is contained in $e_sTe^*_0$.
We now show that the left-hand side is contained in the right-hand side.
By Theorem \ref{thm:span} $e_sTe^*_0$ is spanned by the zigzag words in $T$
that begin with $e_s$ and end with $e^*_0$.
Let $x_1x_2 \cdots x_n$ denote such a word.
We show that $x_1x_2 \cdots x_n$ is contained in the right-hand side of
\eqref{eq:aux}.
By construction $x_n=e^*_0$ so $\ox_n = 0$.
Define
\[
   j = \min\{i \,|\, 1 \leq i \leq n,\; \ox_i = 0\}.
\]
Observe $\ox_j =0$, so $x_j=e_0$ or $x_j = e^*_0$.
First assume $x_j=e_0$, and abbreviate $w = x_1x_2 \cdots x_j$.
By construction $w$ is a nonredundant lifting word that begins
with $e_s$ and ends with $e_0$. 
Observe that $x_1 x_2 \cdots x_n$ = $w x_jx_{j+1} \cdots x_n$ and
$x_jx_{j+1} \cdots x_n \in e_0 T e^*_0$, so $x_1 x_2 \cdots x_n \in w e_0Te^*_0$.
By these comments $x_1 x_2 \cdots x_n$ is contained in the right-hand side of 
\eqref{eq:aux}.
Next assume $x_j=e^*_0$, and abbreviate $w' = x_1x_2 \cdots x_j$.
By construction $w'$ is a nonredundant lifting word that begins with $e_s$ 
and ends with $e^*_0$. 
Observe that $x_1 x_2 \cdots x_n = w' x_jx_{j+1} \cdots x_n$ and
$x_jx_{j+1} \cdots x_n \in e^*_0 T e^*_0$, 
so $x_1 x_2 \cdots x_n \in w' e^*_0Te^*_0$.
By these comments $x_1 x_2 \cdots x_n$ is contained in the right-hand side 
of \eqref{eq:aux}. The result follows..
\end{proof}

\begin{proofof}{Theorem \ref{thm:main}}
Let $A,A^*$ denote the tridiagonal pair in question, and let 
$(A;\{E_i\}_{i=0}^d;A^*$; $\{E^*_i\}_{i=0}^d)$ denote an associated 
tridiagonal system. Let $V$ denote the underlying vector space. 
Following Lemma \ref{lem:T-module} we view $V$ as an irreducible $T$-module, 
where $T$ is from that lemma. 
For $0 \leq s \leq d$ we show $\rho_s \leq \rho_0 \binom{d}{s}$.
To this end we apply each term in \eqref{eq:aux} to $V$, and evaluate 
the results as follows.
Observe $V=Te^*_0V$ since the $T$-module $V$ is irreducible and 
$e^*_0V \not=0$. Therefore $e_sV=e_sTe^*_0V$. 
By construction $e_0Te^*_0V \subseteq e_0V$ and 
$e^*_0Te^*_0V \subseteq e^*_0V$. 
This yields
\begin{equation}     \label{eq:star}
     e_s V \subseteq \sum_{w} w e_0V + \sum_{w'} w' e^*_0V,
\end{equation}
where the first sum (resp. second sum) is over all the nonredundant lifting 
words $w$ (resp. $w'$) that begin with $e_s$ and end with $e_0$ 
(resp. end with $e^*_0$).
In \eqref{eq:star} we consider the dimension of each side.
The dimension of $e_sV$ is $\rho_s$.
In the right-hand side of \eqref{eq:star} there are a total of
$\binom{d}{s}$ summands, by Lemma \ref{lem:ds}.
Each summand has dimension at most $\rho_0$, since $e_0V$ and $e^*_0V$ have 
dimension $\rho_0$.
Therefore in \eqref{eq:star} the right-hand side has dimension at most
$\rho_0 \binom{d}{s}$. It follows that $\rho_s \leq \rho_0 \binom{d}{s}$.
\end{proofof}

\section{Remarks}
\label{sec:conj}

\indent
In this section we give some remarks and suggestions for future research.

\medskip

Fix a nonnegative integer $d$. 
Fix a sequence
$(\{\th_i\}_{i=0}^d$; $\{\th^*_i\}_{i=0}^d)$ of scalars
taken from $\K$ that satisfies \eqref{eq:distinct} and
Lemma {\rm \ref{lem:corec}}.
Let $T$ denote the corresponding $\K$-algebra from Definition \ref{def:T}.
Observe that $e^*_0Te^*_0$ is a $\K$-algebra with multiplicative
identity $e^*_0$.  
In \cite[Theorem 2.6]{NT:structure} we proved that this algebra is generated by
$e^*_0De^*_0$, where $D$ is from Definition \ref{def:D}. 
We now give an alternative proof of this fact using Theorem \ref{thm:span}.
The following notation will be useful. For subsets $X$, $Y$ of $T$ 
let $XY$ denote the $\K$-subspace of $T$ spanned by
$\{xy \,|\, x \in X,\, y \in Y \}$.

\begin{lemma}  {\rm \cite[Proposition 4.6]{NT:structure}} 
 \label{lem:eDDDe}    \samepage
For an integer $n\geq 1$ the space
\begin{equation}   \label{eq:space2}
  e^*_0 D D^* D D^* D \cdots D D^* D e^*_0
    \qquad\qquad (\text{$n$ $D$'s})
\end{equation}
is equal to
\begin{equation}   \label{eq:space3}
  e^*_0 D e^*_0 D e^*_0 D \cdots D e^*_0 D e^*_0
        \qquad\qquad (\text{$n$ $D$'s}).
\end{equation}
\end{lemma}

\begin{proof}
By construction the space \eqref{eq:space3} is contained in the space 
\eqref{eq:space2}.
We show that the space \eqref{eq:space2} is contained in the space 
\eqref{eq:space3}.
The space \eqref{eq:space2} is spanned by the words of length $2n+1$ that
begin and end with $e^*_0$. By Theorem \ref{thm:span} this space is spanned
by the zigzag words of length $2n+1$ that begin and end with $e^*_0$.
Let $x_1x_2\ldots x_{2n+1}$ denote such a zigzag word. 
Using Theorem \ref{thm:p} we find $\ox_i=0$ for all odd 
$i$ $(1 \leq i \leq 2n+1)$.
Therefore  $x_1x_2\ldots x_{2n+1}$ is contained in \eqref{eq:space3} and
the result follows.
\end{proof}

\begin{corollary}    {\rm \cite[Theorem 2.6]{NT:structure}}
The algebra $e^*_0Te^*_0$ is generated by $e^*_0De^*_0$.
\end{corollary}

\begin{proof}
Immediate from Lemma \ref{lem:eDDDe} and since $T$ is generated by
$D$, $D^*$.
\end{proof}

\medskip

We now give some suggestions for future research.

\medskip

\begin{conjecture}    \samepage
For an integer $n\geq 1$,
each of the following sets are linearly independent.
\begin{itemize}
\item[(i)]
The zigzag words of length $n$ in $T$, that begin with a nonstarred element.
\item[(ii)]
The zigzag words of length n in T, that begin with a starred element.
\end{itemize}
\end{conjecture}

\begin{definition}   \samepage
A word $x_1x_2\cdots x_n$ in $T$ is said to be {\em nonrepeating} whenever 
$\ox_{i-1} \not= \ox_i$ for $2 \leq i \leq n$.
\end{definition}

\begin{conjecture}    \samepage
Each of the following is a basis for the $\K$-vector space $T$.
\begin{itemize}
\item[(i)]
The nontrivial nonrepeating zigzag words that begin with a nonstarred element.
\item[(ii)] 
The nontrivial nonrepeating zigzag words that begin with a starred element.
\item[(iii)]
The nontrivial nonrepeating zigzag words that end with a nonstarred element.
\item[(iv)]
The nontrivial nonrepeating zigzag words that end with a starred element.
\end{itemize}
\end{conjecture}

\medskip

We mention a stronger version of Theorem \ref{thm:main}.

\medskip

\begin{conjecture}   \label{conj:shape}    \samepage
Let $\{\rho_i\}_{i=0}^d$ denote the shape of a tridiagonal pair. 
Then there exists a nonnegative integer $N$ and positive integers 
$d_1, d_2, \ldots, d_N$ such that
\[
  \sum_{i=0}^d \rho_i z^i = \rho_0 \prod_{j=1}^N (1+z+z^2 +\cdots + z^{d_j}).
\]
Here $z$ denotes an indeterminate.
\end{conjecture}

\medskip

See \cite[Theorem 9.1]{IT:nonnil},
\cite[Theorem 8.3]{IT:Krawt},
\cite[Theorem 3.3]{N:refine}
for partial results on Conjecture \ref{conj:shape}.

\section{Acknowledgements}

\indent
The second author gratefully acknowledges many conversations with 
Tatsuro Ito (Ka\-na\-za\-wa University) on the general subject of this paper. 
The resulting insights lead directly to the paper, and consequently
we feel that Ito deserves to be a coauthor.
We offered him this coauthorship but he declined.

\bigskip

\noindent Kazumasa Nomura \hfil\break
\noindent College of Liberal Arts and Sciences \hfil\break
\noindent Tokyo Medical and Dental University \hfil\break
\noindent Kohnodai, Ichikawa, 272-0827 Japan \hfil\break
\noindent email: {\tt knomura@pop11.odn.ne.jp} \hfil\break


\bigskip

\noindent Paul Terwilliger \hfil\break
\noindent Department of Mathematics \hfil\break
\noindent University of Wisconsin \hfil\break
\noindent 480 Lincoln Drive \hfil\break
\noindent Madison, WI 53706-1388 USA \hfil\break
\noindent email: {\tt terwilli@math.wisc.edu }\hfil\break

\end{document}